\documentclass[11pt,amsfonts]{article}
\usepackage{graphicx}
\usepackage{latexsym}
\usepackage{amssymb}
\usepackage{amsmath}
\usepackage{layout}
\newtheorem{prop}{Proposition}
\newtheorem{lemma}{Lemma}
\newtheorem{definition}{Definition}

\newtheorem{theorem}{Theorem}
\newtheorem{remark}{Remark}

\def\real{{\mathord{{\rm I\kern-2.8pt R}}}}        
\def\inte{{\mathord{{\rm I\kern-2.8pt N}}}}

\def\sZZ{{\rm Z\kern-2.8ptem{}Z}}

\def\z{{\mathchoice
  {\sZZ}
  {\sZZ}
  {\rm Z\kern-0.30em{}Z}
  {\rm Z\kern-0.25em{}Z} }}
\def\sQQ{{\kern 0.27em \vrule height1.45ex width0.03em depth0em
          \kern-0.30em \rm Q}}
\def\qu{{\mathchoice
    {\sQQ}
    {\sQQ}
  {\kern 0.225em \vrule height1.05ex width0.025em depth0em \kern-0.25em \rm Q}
  {\kern 0.180em \vrule height0.78ex width0.020em depth0em \kern-0.20em \rm Q}
        }}
\def\sCC{{\kern 0.27em \vrule height1.45ex width0.03em depth0em
          \kern-0.30em \rm C}}
\def\complex{{\mathchoice
    {\sCC}
    {\sCC}
  {\kern 0.225em \vrule height1.05ex width0.025em depth0em \kern-0.25em \rm C}
  {\kern 0.180em \vrule height0.78ex width0.020em depth0em \kern-0.20em \rm C}
        }}

\newcommand{\R}{\mathbb{R}}

\newcommand{\ba}{\begin{array}}
\newcommand{\ea}{\end{array}}
\newcommand{\be}{\begin{equation}}
\newcommand{\ee}{\end{equation}}
\newcommand{\bea}{\begin{eqnarray}}
\newcommand{\eea}{\end{eqnarray}}
\newcommand{\beaa}{\begin{eqnarray*}}
\newcommand{\eeaa}{\end{eqnarray*}}

\def\z{\zeta}

\font\tenmath=msbm10 \font\sevenmath=msbm7 \font\fivemath=msbm5
\newfam\mathfam \textfont\mathfam=\tenmath
\scriptfont\mathfam=\sevenmath \scriptscriptfont\mathfam=\fivemath

\def \={{\buildrel {\rm (law)} \over =}}

\def \R{{\math R}}

\def\L{\Lambda}

\def\cN{{\cal N}}

\def\qed{ \hfill \vrule width.25cm height.25cm depth0cm\smallskip}

\newcommand{\basa}{\begin{assumption}}
\newcommand{\easa}{\end{assumption}}

\newcommand{\bas}{\begin{assum}}
\newcommand{\eas}{\end{assum}}

\def\L{\left}
\def\R{\right}

\def\inte{\mathbb{N}}

\newcommand{\ignore}[1]{}
\begin{document}

\renewcommand{\thefootnote}{\fnsymbol{footnote}}

\renewcommand{\thefootnote}{\fnsymbol{footnote}}

\author{Anthony R\'eveillac\footnote{Institut f\"ur Mathematik, Humboldt-Universit\"at zu Berlin, Unter den Linden 6, 10099 Berlin, Germany, areveill@mathematik.hu-berlin.de} \quad Michael Stauch\footnote{Institut f\"ur Mathematik, Humboldt-Universit\"at zu Berlin, Unter den Linden 6, 10099 Berlin, Germany, stauch@mathematik.hu-berlin.de} \quad Ciprian A. Tudor\footnote{Universit\'e de Lille 1, Laboratoire Paul Painlev\'e, Ciprian.Tudor@math.univ-lille1.fr. Associate member: Samm, Universit\'e de Panth\'eon-Sorbonne Paris 1, 90, rue de Tolbiac, Paris 13, France.}}

\title{Hermite variations of the fractional Brownian sheet}

\maketitle

\vspace{-0.5cm}
\begin{center}
\small{\emph{Dedicated to Paul Malliavin}}
\end{center}

\begin{abstract}
We prove central and non-central limit theorems for the Hermite variations of the anisotropic fractional Brownian sheet $W^{\alpha, \beta}$
with Hurst parameter $(\alpha, \beta) \in (0,1)^2$. When $0<\alpha \leq 1-\frac{1}{2q}$ or $0<\beta \leq 1-\frac{1}{2q}$ a central limit theorem holds for the renormalized Hermite variations of order $q\geq 2$, while for $1-\frac{1}{2q}<\alpha, \beta < 1$ we prove that these variations satisfy a non-central limit theorem. In fact, they converge to a random variable which is the value of a two-parameter Hermite process at time $(1,1)$.
\end{abstract}

\vskip0.3cm

{\bf 2010 AMS Classification Numbers:}  60F05, 60H05, 91G70.

 \vskip0.3cm


\section{Introduction}

In recent years a lot of attention has been given to the study of
the (weighted) power variations for stochastic processes. Let us recall the case  of the fractional Brownian motion
(fBm). Consider  $B^H:=(B^H_t)_{t\in [0,1]}$  a fBm with Hurst
parameter $H$ in $(0,1)$. Recall that $B^{H} $ is a centered
Gaussian process with covariance $R_{H}(t,s)=\mathbb{E}(B^{H}_{t}B^{H}_{s})
=\frac{1}{2}(t^{2H}+s^{2H}-\vert t-s\vert ^{2H})$ for every $s,t\in [0,1]$. It can also be
defined as the only self-similar Gaussian process with stationary
increments.
 The  weighted Hermite variations of order $q\geq 1$ of $B^{H} $  are
defined as
$$V_N:=\sum_{i=1}^{N} f(B^H_{(i-1)/N}) H_q(N^{H}
(B^H_{i/N}-B^H_{(i-1)/N})),$$ where $H_q$ denotes the Hermite
polynomial of order $q$ (see Section \ref{section:Malliavin}) and $f$ is a real-valued deterministic function
regular enough. Take for example $q=2$. Since the second-order Hermite
polynomial is $H_{2}(x)=\frac{x^{2}-1}{2}$, the latter quantity is equal to
$\frac12 \sum_{i=0}^{N-1} f(B^H_{(i-1)/n}) (N^{2H}
|B^H_{i/N}-B^H_{(i-1)/N}|^2-1)$. The asymptotic behavior of these
variations plays an important role in estimating the parameter of the
fractional Brownian motion or of other self-similar processes (see e.g.
\cite{coeur} or  \cite{TV}). Weighted Hermite
variations are also crucial in the study of numerical
schemes for stochastic differential equations driven by a fBm (see
\cite{NN}).
A full understanding of $V_n$ is given in
\cite{BreuerMajor,DobrushinMajor,GiraitisSurgailis,Taqqu}, when
$f\equiv 1$ and in \cite{NNT,NourdinReveillac} for quite general
functions $f$. Let us recall the main results for the case $f\equiv
1$:
\begin{itemize}
\item If $0<H<1-1/(2q)$ then, $\quad N^{-1/2} V_N \overset{law}{\underset{n\to\infty}{\longrightarrow}} \mathcal{N}(0,C_H)$,
\item If $H=1-1/(2q)$ then, $\quad (\log(N) N)^{-1/2} V_N \overset{law}{\underset{n\to\infty}{\longrightarrow}} \mathcal{N}(0,C_{H})$,
\item If $1-1/(2q)<H<1$ then, $\quad N^{1-2H} V_N \overset{law}{\underset{n\to\infty}{\longrightarrow}} \textrm{ Hermite r.v.}$
\end{itemize}
Here $C_{H}$ is an explicit positive constant. A Hermite
random variable is the value at time $1$ of a Hermite process, which is a non-Gaussian self-similar
process with stationary increments living in the $q$th Wiener chaos (see e.g. \cite{NNT}).\\
In this paper we use Malliavin calculus and multiple stochastic
integrals to study the asymptotic behavior of the non-weighted
(\textit{i.e.} $f\equiv 1$) Hermite variations, where the fBm
is replaced by a fractional Brownian sheet (fBs), which is a centered Gaussian process
$(W^{\alpha,\beta}_{(s,t)})_{(s,t) \in [0,1]^2}$ whose covariance function is the product of the one of a fBm of parameter $\alpha$
in one direction and of the covariance of a  fBm of parameter $\beta$ in the other component.
We define the Hermite variations based on the rectangular increments
of $W^{\alpha , \beta }$ by, for every $N,M\geq 1$,
  \begin{equation*}
V_{N,M} := \sum_{i=0} ^{N-1}\sum_{j=0}^{M-1}H_{q} \left(
N^{\alpha}M^{\beta} \left(W^{\alpha, \beta } _{\frac{i+1}{N},
\frac{j+1}{M} } -W^{\alpha, \beta } _{\frac{i}{N}, \frac{j+1}{M}
}-W^{\alpha, \beta } _{\frac{i+1}{N}, \frac{j}{M} }+W^{\alpha, \beta
} _{\frac{i}{N}, \frac{j}{M} }\right)\right).
\end{equation*}
In some  sense the fBs is the tensorization of two orthogonal fBm. In the
view of results from the one-dimensional case mentioned above it
would be natural to expect the limit of the correctly
renormalized non-weighted Hermite variations to be the "tensorization" of the limits appearing in the one-parameter
  case. This actually is true in several cases: when $\alpha ,
  \beta \leq 1-\frac{1}{2q}$, the limit in distribution of the renormalized sequence
  $V_{N,}$ is, as expected, a Gaussian random variable and when $\alpha ,
  \beta > 1-\frac{1}{2q}$ the limit (in $L^{2}$ actually) of
  $V_{N,}$  is the value at time $(1,1)$ of a two-parameter Hermite
  process (which will be introduced later in the paper). The most
  interesting and quite unexpected case is when one Hurst parameter is
  less and the other one is strictly bigger than the critical value $1-\frac{1}{2q}$. It turns out that in this situation the limit in
  distribution of the renormalized Hermite variations is still
  Gaussian.\\
 We prove our   central limit theorems using Malliavin calculus and
 the so-called Stein's method on Wiener chaos introduced by Nourdin
 and Peccati in \cite{NoPe1}. Using these results, it
 is actually possible to measure the distance between the law of an
 arbitrary random variable $F$ (differentiable in the sense of the
 Malliavin calculus) and the standard normal law. This distance can
 be bounded by a quantity which involves the Malliavin derivative of
 $F$. Using these tools and analyzing the Malliavin derivatives of
 $V_{N,M}$ (which is an element of the $q$ th Wiener chaos
 generated by the fBs $W^{\alpha, \beta}$) we are able to
 derive a Berry-Ess\'een bound in our central limit theorem.\\

We proceed as follows. In Section \ref{section:pre} we define the
Hermite variations of a fBs and give the basic tools of Malliavin
calculus for the fractional Brownian sheet needed throughout the paper.  The central case is
presented in Section \ref{section:central}, whereas Section
\ref{section:non-central} is devoted to the non-central case. The Appendix contains some auxiliary technical lemmas.

\section{Preliminaries}
\label{section:pre}

\subsection{The fractional Brownian sheet}

Several extensions of the fractional Brownian motion have been proposed in the literature as for example the \textit{fractional Brownian field} (\cite{Lindstrom,BonamiEstrade}), the \textit{L\'evy's fractional Brownian field} (\cite{CiesielskiKamont}) and the \textit{anisotropic fractional Brownian sheet} (\cite{Kamont,AyacheLegerPontier2002}), which we consider in this paper. The definitions and properties of this section can be found in \cite{AlosMazetNualart2001,TudorViens2003}. We begin with the definition of the anisotropic fractional Brownian sheet.

\begin{definition}[Fractional Brownian sheet]
A fractional Brownian sheet $(W_{s,t}^{\alpha,\beta})_{(s,t)\in[0,1]^2}$ with Hurst indices $(\alpha,\beta)\in (0,1)^2$ is a centered two-parameter Gaussian process equal to zero on the set
$$\{(s,t) \in [0,1]^2, \; s=0 \textrm{ or } t=0 \}. $$
For $s_{1},t_{1},s_{2},t_{2}\in [0,1]$ the covariance function is given by
\begin{eqnarray*}
R^{\alpha,\beta}((s_1,t_1),(s_2,t_2))&:=&\mathbb{E}\L[ W_{s_1,t_1}^{\alpha,\beta} W_{s_2,t_2}^{\alpha,\beta} \R]\\
&=&K^\alpha(s_1,s_2) K^\beta(t_1,t_2)\\
&=&\frac12 \big( s_1^{2\alpha}+s_2^{2\alpha}-\vert s_1-s_2 \vert^{2 \alpha} \big) \frac12 \big( t_1^{2 \beta}+t_2^{2 \beta}-\vert t_1-t_2 \vert^{2 \beta} \big).
\end{eqnarray*}
\end{definition}
\noindent
We assume that $(W_{s,t}^{\alpha,\beta})_{(s,t)\in[0,1]^2}$ is defined on a complete probability space $\L(\Omega,\mathcal{F},P\R)$ where $\mathcal{F}$ is generated by $W^{\alpha,\beta}$. Let us denote by $\mathcal{H}^{\alpha,\beta}$ the canonical Hilbert space generated by the Gaussian process $W^{\alpha, \beta}$  defined as the closure of the linear span generated by the indicator functions on $[0,1]^2$ with respect to the scalar product
$$ \langle \textbf{1}_{[0,s_1]\times[0,t_1]}, \textbf{1}_{[0,s_2]\times[0,t_2]} \rangle_{\mathcal{H}^{\alpha,\beta}} = R^{\alpha,\beta}((s_1,t_1),(s_2,t_2)).$$
The mapping $\textbf{1}_{[0,s]\times[0,t]}\mapsto W_{s,t}^{\alpha,\beta}$ provides an isometry between $\mathcal{H}^{\alpha,\beta}$ and the first chaos of $W^{\alpha , \beta }$ denoted by $H_1^{\alpha,\beta}$. For an element $\varphi$ of $\mathcal{H}^{\alpha,\beta}$ we denote by $W^{\alpha,\beta }(\varphi)$ the image of $\varphi$ in the space $H_1^{\alpha,\beta}$.\\
For any $\gamma$ in $(0,1)$ we denote by $\mathcal{H}^\gamma$ the Hilbert space defined as the closure of the linear span generated by indicator functions on $[0,1]$ with respect to the scalar product
$$ \langle \textbf{1}_{[0,s_1]}, \textbf{1}_{[0,s_2]} \rangle_{\mathcal{H}^{\gamma}} = \frac12 \big( s_1^{2\gamma}+s_2^{2\gamma}-\vert s_1-s_2 \vert^{2 \gamma} \big).$$
The space ${\cal{H}}^{\gamma}$ is in fact the canonical Hilbert space generated by the (one-dimensional) fractional Brownian motion with Hurst parameter $\gamma \in (0,1)$.  With these notations we will often use the practical relation
$$ \langle \textbf{1}_{[0,s_1]\times[0,t_1]}, \textbf{1}_{[0,s_2]\times[0,t_2]} \rangle_{\mathcal{H}^{\alpha,\beta}}=\langle \textbf{1}_{[0,s_1]}, \textbf{1}_{[0,s_2]} \rangle_{\mathcal{H}^{\alpha}} \langle \textbf{1}_{[0,t_1]}, \textbf{1}_{[0,t_2]} \rangle_{\mathcal{H}^{\beta}}, \quad \forall (s_1,s_2,t_1,t_2) \in [0,1]^4.$$
More generally, for any two functions $f,g\in {\mathcal{H}^{\alpha,\beta}}$ such that $\int_{[0,1]^{4} }\left| f(u,v)g(a,b)\right| \vert u-a\vert ^{2\alpha -2} \vert v-b\vert ^{2\beta -2} \mathrm{d}a \mathrm{d}b\mathrm{d}u\mathrm{d}v<\infty$ we have
\begin{equation}
\label{scaH}
\langle f,g\rangle_{{\mathcal{H}^{\alpha,\beta}}}=a(\alpha) a(\beta) \int_{[0,1]^{4} } f(u,v)g(a,b) \vert u-a\vert ^{2\alpha -2} \vert v-b\vert ^{2\beta -2} \mathrm{d}a \mathrm{d}b\mathrm{d}u\mathrm{d}v
\end{equation}
with $a(\alpha)=\alpha (2\alpha -1)$. Note finally that we can also give a representation of $(W_{s,t}^{\alpha,\beta})_{(s,t)\in[0,1]^2}$ as a stochastic integral of kernels $K^\alpha$ and $K^\beta$ with respect to a standard Brownian sheet $(W_{(s,t)})_{(s,t)\in[0,1]^2}$:
$$ W_{(s,t)}^{\alpha,\beta}=\int_0^s \int_0^t K^\alpha(s,u) K^\beta(t,v) \; dW_{(u,v)},\quad (s,t)\in[0,1]^2,$$
where $K^{\alpha}$ is the usual kernel of the fractional Brownian motion $B^{\alpha}$ which appears in its expression as a Wiener integral $B^{\alpha}_{t}= \int_{0}^{t} K^{\alpha }(t,s)dW_{s}$ (see e.g. \cite{Nualart3} for an explicit definition of this kernels; we will not use it in this paper). Using this representation, Tudor and Viens in \cite{TudorViens2003,TudorViens2006} have developed a Malliavin calculus with respect to $W^{\alpha,\beta}$. \\
Let us recall the notion of self-similarity and stationary increments for a two-parameter process (see \cite{AyacheLegerPontier2002}).

\begin{definition}\label{stationary}
A two-parameter stochastic process $(X_{s,t})_{(s,t) \in T}$, $T\subset \mathbb{R}^{2}$ has stationary increments if for every $n\in \mathbb{N}$ and for every $(s_{1},t_{1}),(s_{2},t_{2}),\ldots , (s_{n}, t_{n}) \in T$ the law of the vector
\begin{equation*}
\left( X_{s+s_{1}, t+t_{1}},X_{s+s_{2}, t+t_{2}}, \ldots , X_{s+s_{n}, t+t_{n}}\right)
\end{equation*}
does not depends on $(s,t)\in T$.
\end{definition}

\begin{definition}
\label{self}
A two-parameter stochastic process $(X_{s,t})_{(s,t) \in T}$, $T\subset \mathbb{R}^{2}$  is self-similar with order $(\alpha, \beta)$ if for any $h,k>0$ the process $(\widehat{X}_{s,t} )_{s,t\in T}$
\begin{equation*}
\widehat{X}_{s,t}:= h^{\alpha } k^{\beta } X _{ \frac{s}{h}, \frac{t}{k} }
\end{equation*}
has the same law as the process $X$.
\end{definition}
Note that the fractional Brownian sheet $W^{\alpha, \beta}$ is self-similar and has stationary increments in the sense of Definitions \ref{self} and \ref{stationary} (see  \cite{AyacheLegerPontier2002}).\\
Now we present some elements of Malliavin calculus with respect to fractional Brownian sheets and especially the Malliavin integration by parts formula (\ref{eq:generalizedMallavinIBP}).

\subsection{Malliavin calculus for the fractional Brownian sheet}
\label{section:Malliavin}
We recall some definitions and properties of the Malliavin calculus for the fractional Brownian sheet. For general Gaussian processes these are contained in the framework described in \cite{Nualart3}.\\\\
\noindent
By $\mathcal{C}_b^\infty(\real^n)$ we denote the space of infinitely differentiable functions from $\real^n$ to $\real$ with bounded partial derivatives. For a cylindrical functional $F$ of the form
\begin{equation}
\label{eq:cylindric}
F=f\L(W^{\alpha,\beta}(\varphi_1),\ldots,W^{\alpha,\beta}(\varphi_n)\R), \quad n\geq 1, \; \varphi_1,\ldots,\varphi_n \in \mathcal{H}^{\alpha,\beta}, \; f \in \mathcal{C}_b^\infty(\real^n),
\end{equation}
we define the Malliavin derivative $DF$ of $F$ as,
$$ DF:=\sum_{i=1}^n \partial_i f\L(W^{\alpha,\beta}(\varphi_1),\ldots,W^{\alpha,\beta}(\varphi_n)\R) \varphi_i.$$
Furthermore $D:L^2(\Omega,\mathcal{F},P) \to L^2(\Omega,\mathcal{F},P;\mathcal{H}^{\alpha,\beta})$ is a closable operator and it can be extended to the closure of the Sobolev space $\mathbb{D}^{1,2}$ defined by the functionals $F$ whose norm $\|F\|_{1,2}$ is finite with
$$\|F\|_{1,2}^{2}:=\mathbb{E}\L[ F^2 \R] + \mathbb{E}\L[\| DF \|_{\mathcal{H}^{\alpha,\beta}}^2\R].$$
The adjoint operator $I_1$ of $D$ is called the divergence operator and it is  defined by the following duality relationship
$$ \mathbb{E}\L[ F I_1(u) \R] = \mathbb{E}\L[ \langle DF, u \rangle_{\mathcal{H}^{\alpha,\beta}} \R],$$
for $F$ in $\mathbb{D}^{1,2}$ and for $u$ in $\mathcal{H}^{\alpha,\beta}$ such that there exists a constant $c_u>0$, satisfying
$$ \L\vert \mathbb{E}\L[ \langle DG, u \rangle_{\mathcal{H}^{\alpha,\beta}} \R] \R\vert \leq c_u \|G\|_{L^2(\Omega,\mathcal{F},P)},\quad \textrm{ for every functional } G \textrm{ of the form } (\ref{eq:cylindric}).$$
Let $n\geq 1$. The $n$th Wiener chaos $\mathfrak{H}_n$ of $W^{\alpha,\beta}$ is the closed linear subspace of $L^2(\Omega,\mathcal{F},P)$ generated by the random variables $\L\{H_n\L(W^{\alpha,\beta}(\varphi)\R), \; \varphi \in \mathcal{H}^{\alpha,\beta}, \; \|\varphi\|_{\mathcal{H}^{\alpha,\beta}}=1\R\}$ where $H_n$ denotes the $n$th Hermite polynomial $$H_{n}(x)=\frac{(-1)^{n}}{n!}e^{\frac{x^{2}}{2}}\frac{d^{q}}{dx^{q}}\left(e^{-\frac{x^{2}}{2}}\right) .$$ A linear isometry between the symmetric tensor product $\left( {\mathcal{H}^{\alpha,\beta}}\right)^{\odot n}$ and $\mathfrak{H}_n$ is defined as
\begin{equation}
\label{eq:I_nH_n}
I_n\L(\varphi^{\otimes n}\R) := n! H_n\L(W^{\alpha,\beta}(\varphi)\R).
\end{equation}
We conclude this section by the following integration by parts formula:
\begin{equation}
\label{eq:generalizedMallavinIBP}
\mathbb{E}\L[ F I_n(h) \R] = \mathbb{E}\L[ \langle D^n F, h \rangle_{\left( {\mathcal{H}^{\alpha,\beta}}\right) ^{\otimes n}} \R], \quad h\in \left( {\mathcal{H}^{\alpha,\beta}}\right) ^{\odot n}, \; F \in \mathbb{D}^{n,2},
\end{equation}
where $\mathbb{D}^{n,2}$ is the space of functionals $F$ such that $\|F\|_{n,2}$ is finite with
$$ \|F\| ^{2}_{n,2}:=\mathbb{E}\L[ F^2 \R] + \sum_{i=1}^n \mathbb{E}\L[\| D^i F \|_{\mathcal{H}^{\alpha,\beta}}^2\R]. $$

\subsection{Hermite variations of the fBs}

Let $(W^{\alpha, \beta}_{s,t})_{s,t\geq 0}$ be a fractional Brownian sheet with Hurst parameter $(\alpha, \beta) \in (0,1)^2$.
We will define the Hermite variations of order $q\geq 1$ of the fractional Brownian sheet by
\begin{equation}
V_{N,M} := \sum_{i=0} ^{N-1}\sum_{j=0}^{M-1}H_{q} \left( N^{\alpha}M^{\beta} \left(W^{\alpha, \beta } _{\frac{i+1}{N}, \frac{j+1}{M} } -W^{\alpha, \beta } _{\frac{i}{N}, \frac{j+1}{M} }-W^{\alpha, \beta } _{\frac{i+1}{N}, \frac{j}{M} }+W^{\alpha, \beta } _{\frac{i}{N}, \frac{j}{M} }\right)\right),
\label{eq:V_nm}
\end{equation}
where $H_{q}$ is the Hermite polynomial of order $q$.  Note that
\begin{equation*}
\mathbb{E}\left(  W^{\alpha, \beta } _{\frac{i+1}{N}, \frac{j+1}{M} } -W^{\alpha, \beta } _{\frac{i}{N}, \frac{j+1}{M} }-W^{\alpha, \beta } _{\frac{i+1}{N}, \frac{j}{M} }+W^{\alpha, \beta } _{\frac{i}{N}, \frac{j}{M} } \right) ^{2}=N^{-2\alpha} M^{-2\beta},
\end{equation*}
which explains the appearance of the factor $N^{\alpha} M^{\beta}$ in (\ref{eq:V_nm}): with this factor the random variable $N^{\alpha}M^{\beta} \left(W^{\alpha, \beta } _{\frac{i+1}{N}, \frac{j+1}{M} } -W^{\alpha, \beta } _{\frac{i}{N}, \frac{j+1}{M} }-W^{\alpha, \beta } _{\frac{i+1}{N}, \frac{j}{M} }+W^{\alpha, \beta } _{\frac{i}{N}, \frac{j}{M} }\right)$ has $L^{2}$-norm equal to 1.\\
We will use  the notation
\begin{equation*}
\Delta i= \left[\frac{i}{N}, \frac{i+1}{N}\right] \mbox{ and } \Delta i,j= \left[\frac{i}{N}, \frac{i+1}{N}\right]\times \left[\frac{j}{M}, \frac{j+1}{M}\right] =\Delta i\times \Delta j,
\end{equation*}
for $i\in\{0,\dots,N-1\}, \; j\in\{0,\dots, M-1\}$. In principle $\Delta i = \Delta i ^{(N)}$ depends on $N$ but we will omit the superscript $N$ to simplify the notation.  With this notation we can write
\begin{equation*}
W^{\alpha, \beta } _{\frac{i+1}{N}, \frac{j+1}{M} } -W^{\alpha, \beta } _{\frac{i}{N}, \frac{j+1}{M} }-W^{\alpha, \beta } _{\frac{i+1}{N}, \frac{j}{M} }+W^{\alpha, \beta } _{\frac{i}{N}, \frac{j}{M} }=I_{1} \left( \mathbf{1}_{[\frac{i}{N}, \frac{i+1}{N}]\times [\frac{j}{M}, \frac{j+1}{M}]}\right)=I_1 (\mathbf{1}_{\Delta i,j}) = I_{1}(\mathbf{1}_{\Delta i \times \Delta j}).
\end{equation*}
Here, and throughout the paper,  $I_{n}$ indicates the multiple integral of order $n>1$ with respect to the fractional Brownian sheet $W^{\alpha , \beta}$. Since for any deterministic function $h\in {\cal{H}}^{\alpha , \beta}$ with  norm one we have
\begin{equation*}
H_{q}(I_{1} (h)) =\frac{1}{q!} I_{q} (h^{\otimes q} ),
\end{equation*}
we derive at
\begin{equation*}
V_{N,M}=\frac{1}{q!}  \sum_{i=1} ^{N}\sum_{j=1}^{M}N^{\alpha q} M^{\beta q}I_{q} \left(   \mathbf{1}_{[\frac{i}{N}, \frac{i+1}{N}]\times [\frac{j}{M}, \frac{j+1}{M}]} ^{\otimes q} \right).
\end{equation*}
We want to study the limit of the (suitably normalized) sequence $V_{N,M}$ as $N,M\to \infty$.
 Since this normalization is depending on the choice of $\alpha$ and $\beta$, we will normalize it with a
 function $\varphi(\alpha,\beta,N,M)$. \\
 Let us define
 \begin{equation}\label{tildeV}
\tilde V_{N,M}:=\frac{1}{q!} \varphi(\alpha,\beta,N,M)\sum_{i=1} ^{N}\sum_{j=1}^{M}I_{q} \left(\mathbf{1}_{[\frac{i}{N}, \frac{i+1}{N}]\times [\frac{j}{M}, \frac{j+1}{M}]} ^{\otimes q} \right).
\end{equation}
By renormalization of the sequence $V_{N,M}$ we understand a function $\varphi (\alpha , \beta , N, M)$ to fulfill the property
 $\mathbb{E}\tilde{V}_{N,M} ^{2} \overset{N,M\to \infty}{\to} 1$.\\
It turns out that the limit of the sequence $\tilde{V}_{N,M}$ is
either Gaussian, or a Hermite random variable,
which is the value at time $(1,1)$ of a two-parameter Hermite process.\\
In the case when $\tilde{V}_{N,M} $ converges to a Gaussian random
variable, our proof will be based on the following result (see
\cite{NoPe1}, see also \cite{NoPe3}).

\begin{theorem}
\label{th:NourdinPeccati}
Let $F$  be a random variable in the $q$th Wiener chaos. Then
\begin{eqnarray}
\label{dkolHq}
d(F,N) & \leq & c\sqrt{\mathbb{E}\left(\left(1 - q^{-1}\left\|DF\right\|_{\mathcal{H}^{\alpha,\beta}}^{2}\right)^{2}\right)}
\end{eqnarray}
 The above inequality still holds true for several distances (Kolmogorov, Wasserstein, total variation or Fortet-Mourier). The constant $c$ is equal to 1 in the case of the Kolmogorov and of the Wasserstein distance, $c$=2 for the  total variation distance and $c=4$ in the case of the Fortet-Mourier distance.
\end{theorem}
Until the end of this paper $d$ will denote one of the distances
mentioned in the previous theorem. We will also assume that $q\geq 2$ because for $q=1$ we have $H_{1}=x$ and then $V_{N,M}$ is Gaussian; this case is trivial. Our argumentation has the
following structure. We first compute the Malliavin derivative (with respect to the fractional Brownian sheet $W^{\alpha, \beta }$)
$D\tilde V_{N,M}$ and we compute its norm in the space
${\cal{H}}^{\alpha, \beta }$. We will get
\begin{equation*}
D\tilde V_{N,M}=\frac{1}{(q-1)!} \varphi(\alpha,\beta,N,M)\sum_{i=1}
^{N}\sum_{j=1}^{M}I_{q-1} \left(\mathbf{1}_{[\frac{i}{N},
\frac{i+1}{N}]\times [\frac{j}{M}, \frac{j+1}{M}]} ^{\otimes q-1}
\right)\mathbf{1}_{[\frac{i}{N}, \frac{i+1}{N}]\times [\frac{j}{M},
\frac{j+1}{M}]},
\end{equation*}
and
\begin{eqnarray*}
\Vert D \tilde V_{N,M} \Vert ^{2} _{{\cal{H}}^{\alpha, \beta } }&=&\frac{1}{(q-1)!^2} (\varphi(\alpha,\beta,N,M))^2\times \sum_{i,i'=0}^{N-1} \sum_{j,j'=0}^{M-1}\langle \mathbf{1}_{\Delta i,j}(\cdot ), \mathbf{1}_{\Delta i',j'}(\cdot )\rangle  _{{\cal{H}}^{\alpha, \beta } }\\
&&\times I_{q-1}  \left( \mathbf{1}_{\Delta i,j} ^{\otimes q-1}\right)I_{q-1}
\left( \mathbf{1}_{\Delta i',j'} ^{\otimes q-1}\right).
\end{eqnarray*}
The product formula for multiple integrals (see \cite{Nualart3}, Chapter 1) reads
\begin{eqnarray*}
I_{q-1}  \left( \mathbf{1}_{\Delta i,j} ^{\otimes q-1}\right)I_{q-1}  \left( \mathbf{1}_{\Delta i',j'} ^{\otimes q-1}\right)&=&\sum_{p=0}^{q-1} p! (C_{q-1}^{p})^{2} \langle \mathbf{1}_{\Delta i,j}(\cdot ), \mathbf{1}_{\Delta i',j'}(\cdot )\rangle  _{{\cal{H}}^{\alpha, \beta } }^{p}\\
&&\times I_{2q-2-2p}\left( \mathbf{1}_{\Delta i,j} ^{\otimes
q-1-p}\tilde{\otimes } \mathbf{1}_{\Delta i',j'} ^{\otimes q-1-p}\right),
\end{eqnarray*}
where $C_{q-1}^{p}:=\begin{pmatrix}
                      q-1 \\
                      p \\
                    \end{pmatrix}
$ for $q\geq 2, p\leq q-1$ and $f\tilde{\otimes } g$ denotes the symmetrization of the function $f\otimes g$.   Hence, we have
\begin{eqnarray*}
\Vert D \tilde V_{N,M} \Vert ^{2} _{{\cal{H}}^{\alpha, \beta } }&=&\frac{1}{(q-1)!^2} (\varphi(\alpha,\beta,N,M))^2\times \sum_{i,i'=0}^{N-1} \sum_{j,j'=0}^{M-1}\sum_{p=0}^{q-1}\langle \mathbf{1}_{\Delta i,j}(\cdot ), \mathbf{1}_{\Delta i',j'}(\cdot )\rangle  _{{\cal{H}}^{\alpha, \beta } }^{p+1}p! \\
&&\times(C_{q-1}^{p})^{2}I_{2q-2-2p}\left( \mathbf{1}_{\Delta i,j} ^{\otimes
q-1-p}\tilde{\otimes } \mathbf{1}_{\Delta i',j'} ^{\otimes q-1-p}\right).
\end{eqnarray*}
Let us isolate the term $p=q-1$ in the above expression. In this
case $2q-2-2p=0$ and this term gives the expectation of $\Vert
D \tilde V_{N,M} \Vert ^{2} _{{\cal{H}}^{\alpha, \beta }
}$.
\begin{eqnarray}
&&\Vert D\tilde V_{N,M} \Vert ^{2} _{{\cal{H}}^{\alpha, \beta } }\\
&=&\frac{1}{(q-1)!^2} (\varphi(\alpha,\beta,N,M))^2 \nonumber\\
&&\times \sum_{i,i'=0}^{N-1}
\sum_{j,j'=0}^{M-1}\sum_{p=0}^{q-2}\langle \mathbf{1}_{\Delta i,j}(\cdot ),
\mathbf{1}_{\Delta i',j'}(\cdot )\rangle  _{{\cal{H}}^{\alpha, \beta }
}^{p+1}p!(C_{q-1}^{p})^{2}I_{2q-2-2p}\left( \mathbf{1}_{\Delta i,j} ^{\otimes
q-1-p}\tilde{\otimes }
\mathbf{1}_{\Delta i',j'} ^{\otimes q-1-p}\right)\nonumber\\
&&+\frac{1}{(q-1)!} (\varphi(\alpha,\beta,N,M))^2\sum_{i,i'=0}^{N-1}
\sum_{j,j'=0}^{M-1}\langle \mathbf{1}_{\Delta i,j}(\cdot ), \mathbf{1}_{\Delta
i',j'}(\cdot )\rangle  _{{\cal{H}}^{\alpha, \beta }
}^{q}=:T_1+T_2.\label{eq:T_1+T_2}
\end{eqnarray}
The term $T_{2}$ is a deterministic term which is equal to $\mathbb{E}\Vert D\tilde V_{N,M} \Vert ^{2} _{{\cal{H}}^{\alpha, \beta } }$.\\
With the correct choice of the normalization we will show that
$T_2$ is converging to $q$ as $N,M$ goes to infinity and $T_1$ converges to zero in $L^2$ sense. Using Theorem \ref{th:NourdinPeccati} we will prove the convergence  to a standard normal random variable of $\tilde{V}_{N,M}$ and we  give bounds for the speed of convergence. The distinction between the two cases (when the limit is normal and when the limit is non-Gaussian) will be made by the term $T_{1}$: it converges to zero if $\alpha \leq 1-\frac{1}{2q}$ or $\beta  \leq 1-\frac{1}{2q}$, while for $\alpha , \beta >1-\frac{1}{2q}$ this term converges to a constant.\\
Let us first discuss the normalization $\varphi(\alpha,\beta,N,M)$ and the convergence of $T_2$ in the following lemma. Given two sequences of real numbers $(a_n)_{n \geq 1}$ and $(b_n)_{n \geq 1}$, we write $a_n \unlhd b_n$ for $\sup_{n \geq 1} \frac{|a_n|}{|b_n|} < \infty$.
\begin{lemma}
\label{T_2}
Let $T_2$ be as in \eqref{eq:T_1+T_2}. Then $q^{-1}T_2\overset{N,M\to \infty}{\to} 1$ for the following choices of $\varphi$:
\begin{itemize}
\item[1)] $\varphi(\alpha,\beta,N,M)=\sqrt{\frac{q!}{s_\alpha s_\beta}} N^{\alpha q-1/2}M^{\alpha q-1/2}$, if $0<\alpha,\beta<1-\frac{1}{2q}$ and $q^{-1}T_2-1 \unlhd N^{-1}+N^{2q\alpha-2q+1}+M^{-1}+M^{2q\beta-2q+1}$,\\
\item[2)] $\varphi(\alpha,\beta,N,M)=\sqrt{\frac{q!}{s_\alpha \iota_\beta}} N^{\alpha q-1}M^{q-1}(\log M)^{-1/2}$, if $0<\alpha<1-\frac{1}{2q}, \; \beta=1-\frac{1}{2q}$ and $q^{-1}T_2-1 \unlhd N^{-1}+N^{2q\alpha-2q+1} + (\log M)^{-1}$,\\
\item[3)] $\varphi(\alpha,\beta,N,M)=\sqrt{\frac{q!}{\iota_\alpha \iota_\beta}} N^{q-1}(\log N)^{-1/2} M^{q-1}(\log M)^{-1/2}$, if $\alpha=\beta=1-\frac{1}{2q}$ and $q^{-1}T_2-1 \unlhd (\log N)^{-1}+(\log M)^{-1}$,\\
\item[4)] $\varphi(\alpha,\beta,N,M)=\sqrt{\frac{q!}{s_\alpha \kappa_\beta}} N^{\alpha q-1/2} M^{q-1}$, if $0<\alpha<1-\frac{1}{2q},\; \beta>1-\frac{1}{2q}$ and $q^{-1}T_2-1 \unlhd N^{-1}+N^{-2q\alpha+2q-1}+M^{-2q\beta+2q-1}$\\
\item[5)] $\varphi(\alpha,\beta,N,M)=\sqrt{\frac{q!}{\iota_\alpha \kappa_\beta}} N^{q-1} (\log N)^{-1/2} M^{q-1}$, if $0<\alpha=1-\frac{1}{2q}, \; \beta>1-\frac{1}{2q}$ and $q^{-1}T_2-1 \unlhd (\log(N)^{-1}+M^{-2q\beta+2q-1}$.
 \item[6)] $\varphi(\alpha,\beta,N,M)=\sqrt{\frac{q!}{\kappa_\alpha \kappa_\beta}} N^{q-1}  M^{q-1}$, if $\alpha> 1-\frac{1}{2q}, \; \beta>1-\frac{1}{2q}$ and $q^{-1}T_2-1 \unlhd N^{-2q\alpha+2q-1}+M^{-2q\beta+2q-1}$.
\end{itemize}
where $s_\cdot$, $\iota_\cdot$ and $\kappa_\cdot$ are defined in Lemma \ref{lemma:estimqteforT_2} in the Appendix.
\end{lemma}
\begin{proof}
Using the properties of the scalar product in Hilbert spaces  we have
\begin{eqnarray*}
T_2&=&\frac{1}{(q-1)!} (\varphi(\alpha,\beta,N,M))^2\times \sum_{i,i'=0}^{N-1} \sum_{j,j'=0}^{M-1}\langle \mathbf{1}_{\Delta i,j}(\cdot ), \mathbf{1}_{\Delta i',j'}(\cdot )\rangle  _{{\cal{H}}^{\alpha, \beta } }^{q}\\
&=&\frac{1}{(q-1)!} (\varphi(\alpha,\beta,N,M))^2\times \left(\sum_{i,i'=0}^{N-1} \langle \mathbf{1}_{\Delta i}(\cdot ), \mathbf{1}_{\Delta i'}(\cdot )\rangle_{{\cal{H}}^{\alpha} }^{q} \right) \times \left(\sum_{j,j'=0}^{M-1} \langle \mathbf{1}_{\Delta j}(\cdot ), \mathbf{1}_{\Delta j'}(\cdot )\rangle  _{{\cal{H}}^{\beta} }^{q}\right).
\end{eqnarray*}
The result follows then from Lemma \ref{lemma:estimqteforT_2} in the Appendix. \qed\\\\
\end{proof}

\begin{remark}
As  mentioned above, $T_{2} =\mathbb{E} \Vert D\tilde{V}_{N,M}\Vert _{{\cal{H}}^{\alpha, \beta}} $. On the other hand, we also have
\begin{equation*}
 q T_{2}=\mathbb{E}\tilde{V}_{N,M}^{2}.
 \end{equation*}
 Indeed, this is true because for every multiple integral $F=I_{q}(f)$, it holds that $\mathbb{E}F^{2}= q\mathbb{E} \Vert DF\Vert ^{2}_{{\cal{H}}^{\alpha, \beta}}$.
\end{remark}
\section{The Central Limit Case}\label{section:central}

We will prove that for every $\alpha , \beta \in (0,1)^{2} \setminus \left(1-\frac{1}{2q}, 1\right) ^{2}$ a Central Limit Theorem holds, where $\tilde{V}_{N,M}$ was defined in (\ref{tildeV}). Using the Stein's method we also give the Berry-Ess\'een bounds for this convergence.

\begin{theorem}(Central Limits)\\
Let $\tilde V_{N,M}$ be defined by (\ref{tildeV}). For every
$(\alpha,\beta)\in (0,1) ^{2} $, we denote by $c_{\alpha,\beta}$ a generic positive  constant which
depends on $\alpha,\beta,q$ and on the distance $d$ and which is
independent of $N$ and $M$. We have
\begin{enumerate}
\item[1)] If $0<\alpha,\beta<1-\frac{1}{2q}$, then $\tilde V_{N,M}$ converges in law to a standard normal r.v. $\cN$ with normalization $\varphi(\alpha,\beta,N,M)=\sqrt{\frac{q!}{s_\alpha s_\beta}} N^{\alpha q-1/2}M^{\alpha q-1/2}.$ In addition
$$ d(\tilde V_{N,M},\cN) \leq c_{\alpha,\beta} \sqrt{N^{-1}+N^{2\alpha-2}+N^{2\alpha q-2q+1} + M^{-1}+M^{2\beta-2}+M^{2\beta q-2q+1}}.$$
\item[2)] If $0<\alpha<1-\frac{1}{2q}$ and $\beta=1-\frac{1}{2q}$, then $\tilde V_{N,M}$ converges in law to a standard normal r.v. $\cN$ with normalization $\varphi(\alpha,\beta,N,M)=\sqrt{\frac{q!}{s_\alpha \iota_\beta}} N^{\alpha q-1}M^{q-1}(\log M)^{-1/2}.$ In addition
$$ d(\tilde V_{N,M},\cN) \leq c_{\alpha,\beta} \sqrt{N^{-1}+N^{2\alpha-2}+N^{2\alpha q-2q+1} + (\log{M})^{-1}}.$$
\item[3)] If both $\alpha=\beta=1-\frac{1}{2q}$, then $\tilde V_{N,M}$ converges in law to a standard normal r.v. $\cN$ with normalization $\varphi(\alpha,\beta,N,M)=\sqrt{\frac{q!}{\iota_\alpha \iota_\beta}} N^{q-1}(\log N)^{-1/2} M^{q-1}(\log M)^{-1/2}.$ In addition
$$ d(\tilde V_{N,M},\cN) \leq c_{\alpha,\beta} \sqrt{\log{N}^{-1} + \log{M}^{-1}}.$$
\item[4)] If $\alpha<1-\frac{1}{2q}$ and $\beta>1-\frac{1}{2q}$, then $\tilde V_{N,M}$ converges in law to a standard normal r.v. $\cN$ with normalization $\varphi(\alpha,\beta,N,M)=\sqrt{\frac{q!}{s_\alpha \kappa_\beta}} N^{\alpha q-1/2} M^{q-1}$. In addition
$$ d(\tilde V_{N,M},\cN) \leq c_{\alpha,\beta} \sqrt{N^{-1}+N^{2\alpha-2}+N^{2\beta q-2q+1}+M^{2\beta q-2q+1}}.$$
\item[5)] If $\alpha=1-\frac{1}{2q}$ and $\beta>1-\frac{1}{2q}$, then $\tilde V_{N,M}$ converges in law to a standard normal r.v. $\cN$ with normalization $\varphi(\alpha,\beta,N,M)=\sqrt{\frac{q!}{\iota_\alpha \kappa_\beta}} N^{q-1} (\log N)^{-1/2} M^{q-1}.$ In addition
$$ d(\tilde V_{N,M},\cN) \leq c_{\alpha,\beta} \sqrt{\log(N)^{-1}+M^{2\beta q-2q+1}}.$$
\end{enumerate}
\end{theorem}
\begin{proof}
Recall that
\begin{eqnarray*}
\Vert D\tilde V_{N,M} \Vert ^{2} _{{\cal{H}}^{\alpha, \beta }
}&=&:T_1+T_2,
\end{eqnarray*}
where the summands $T_{1}$ and $T_{2}$ are given as in (\ref{eq:T_1+T_2}). We apply Lemma \ref{T_2} to see that $1-q^{-1}T_2$ converges to zero as $N,M$ goes to infinity.\\
Let us show that $T_1$ is converging to zero in $L^2(\Omega)$. We use the orthogonality of the iterated integrals to compute
\begin{eqnarray*}
\mathbb{E}T_1^2&=&\frac{1}{(q-1)!^4}(\varphi(\alpha,\beta,N,M))^4\sum_{i,i',k,k'=0}^{N-1} \sum_{j,j', \ell, \ell '=0}^{M-1}\sum_{p=0}^{q-2} (p!)^2\\
&&(C_{q-1}^{p})^4\langle \mathbf{1}_{\Delta i,j}(\cdot ), \mathbf{1}_{\Delta i',j'}(\cdot )\rangle  _{{\cal{H}}^{\alpha, \beta } }^{p+1}
\langle \mathbf{1}_{\Delta k,\ell}(\cdot ), \mathbf{1}_{\Delta k',\ell '}(\cdot )\rangle  _{{\cal{H}}^{\alpha, \beta } }^{p+1}\\
&&\times \mathbb{E} \left[I_{2q-2-2p}\left( \mathbf{1}_{\Delta i,j} ^{\otimes q-1-p}\tilde{ \otimes }
\mathbf{1}_{\Delta i',j'} ^{\otimes q-1-p}\right)
I_{2q-2-2p}\left( \mathbf{1}_{\Delta k,\ell} ^{\otimes q-1-p}\tilde{ \otimes }
\mathbf{1}_{\Delta k',\ell'} ^{\otimes q-1-p}\right)\right]\\
&=&\frac{1}{(q-1)!^4}(\varphi(\alpha,\beta,N,M))^4\sum_{i,i',k,k'=0}^{N-1} \sum_{j,j', \ell, \ell '=0}^{M-1}\sum_{p=0}^{q-2} (p!)^2\\
&&(C_{q-1}^{p})^4\langle \mathbf{1}_{\Delta i,j}(\cdot ), \mathbf{1}_{\Delta i',j'}(\cdot )\rangle  _{{\cal{H}}^{\alpha, \beta } }^{p+1}
\langle \mathbf{1}_{\Delta k,\ell}(\cdot ), \mathbf{1}_{\Delta k',\ell '}(\cdot )\rangle  _{{\cal{H}}^{\alpha, \beta } }^{p+1}\\
&&\times (2q-2-2p)! \langle \mathbf{1}_{\Delta i,j} ^{\otimes q-1-p}\tilde{ \otimes }
\mathbf{1}_{\Delta i',j'} ^{\otimes q-1-p}, \mathbf{1}_{\Delta k,\ell} ^{\otimes q-1-p}\tilde{ \otimes }
\mathbf{1}_{\Delta k',\ell '} ^{\otimes q-1-p}\rangle _{{\cal{H}}^{\alpha , \beta} }.\\
\end{eqnarray*}
Now, let us discuss the tensorized terms. We use the fact that
\begin{eqnarray*}
&&\langle \mathbf{1}_{\Delta i,j} ^{\otimes q-1-p}\tilde{ \otimes }
\mathbf{1}_{\Delta i',j'} ^{\otimes q-1-p}, \mathbf{1}_{\Delta k,\ell} ^{\otimes q-1-p}\tilde{ \otimes }
\mathbf{1}_{\Delta k,\ell '} ^{\otimes q-1-p}\rangle _{{\cal{H}}^{\alpha , \beta} }\\
&=&\sum_{a+b=q-1-p; c+d= q-1-p} \langle \mathbf{1}_{\Delta i,j}, \mathbf{1}_{\Delta k,\ell}\rangle ^{a} _{{\cal{H}}^{\alpha, \beta }}\langle \mathbf{1}_{\Delta i,j},\mathbf{1}_{\Delta k',\ell '}\rangle ^{b} _{{\cal{H}}^{\alpha, \beta}}\\
&&\times \langle \mathbf{1}_{\Delta i',j'}, \mathbf{1}_{\Delta k,\ell}\rangle ^{c} _{{\cal{H}}^{\alpha, \beta }}\langle \mathbf{1}_{\Delta i',j'},\mathbf{1}_{\Delta k',\ell '}\rangle ^{d} _{{\cal{H}}^{\alpha, \beta}}\\
&=&\sum_{a+b=q-1-p; c+d= q-1-p} \langle \mathbf{1}_{\Delta i},\mathbf{1}_{\Delta k}\rangle ^{a} _{{\cal{H}}^{\alpha }}
\langle \mathbf{1}_{\Delta j}, \mathbf{1}_{\Delta \ell}\rangle ^{a} _{{\cal{H}}^{\beta}}\\
&&\times
\langle \mathbf{1}_{\Delta i},\mathbf{1}_{\Delta k'}\rangle ^{b} _{{\cal{H}}^{\alpha }}
\langle \mathbf{1}_{\Delta j}, \mathbf{1}_{\Delta \ell '}\rangle ^{b} _{{\cal{H}}^{\beta}}
\langle \mathbf{1}_{\Delta i'},\mathbf{1}_{\Delta k}\rangle ^{c} _{{\cal{H}}^{\alpha }}
\langle \mathbf{1}_{\Delta j'}, \mathbf{1}_{\Delta \ell}\rangle ^{c} _{{\cal{H}}^{\beta}}\\
&&\times
\langle \mathbf{1}_{\Delta i'},\mathbf{1}_{\Delta k'}\rangle ^{d} _{{\cal{H}}^{\alpha }}
\langle \mathbf{1}_{\Delta j'}, \mathbf{1}_{\Delta \ell '}\rangle ^{d} _{{\cal{H}}^{\beta}}
\end{eqnarray*}
(we recall  that $\mathbf{1}_{\Delta i}:=\mathbf{1}_{[\frac{i}{N},\frac{i+1}{N}]}$). Therefore we finally have
\begin{eqnarray*}
\mathbb{E}T_1^2&=&\frac{1}{(q-1)!^4}(\varphi(\alpha,\beta,N,M))^4\sum_{p=0}^{q-2}(C_{q-1}^{p})^4 (p!)^2 \sum_{a+b=q-1-p; c+d= q-1-p}\\
&&a_{N}(p, \alpha, a,b,c,d )b_{M}(p, \beta, a, b,c ,d),
\end{eqnarray*}
with
\begin{eqnarray*}
a_{N}(p, \alpha, a, b, c,d)&=& \sum_{i,i',k,k'=0}^{N-1}\langle \mathbf{1}_{\Delta i},\mathbf{1}_{\Delta k}\rangle ^{a} _{{\cal{H}}^{\alpha }}\langle \mathbf{1}_{\Delta i},\mathbf{1}_{\Delta k'}\rangle ^{b} _{{\cal{H}}^{\alpha }}\\
&&\langle \mathbf{1}_{\Delta i'},\mathbf{1}_{\Delta k}\rangle ^{c} _{{\cal{H}}^{\alpha }}\langle \mathbf{1}_{\Delta i'},\mathbf{1}_{\Delta k'}\rangle ^{d} _{{\cal{H}}^{\alpha }}\\
&&\langle \mathbf{1}_{\Delta i},\mathbf{1}_{\Delta i'}\rangle ^{p+1} _{{\cal{H}}^{\alpha }}\langle \mathbf{1}_{\Delta k},\mathbf{1}_{\Delta k'}\rangle ^{p+1} _{{\cal{H}}^{\alpha }}
\end{eqnarray*}
and $b_{M}(p, \beta, a, b, c,d)$ similarly defined.
We apply Lemma \ref{lemma:estimateforT_1} stated in the Appendix to the terms $a_N$ and $b_M$ to conclude the convergence of $T_1$ to zero. Hence, $\mathbb{E}[(1 - q^{-1}\|D\tilde V_{N,M}\|_{\mathcal{H}^{\alpha,\beta}}^{2})^{2}]=q^{-2} \mathbb{E}[|T_1|^2] +(1-q^{-1} T_2)^2$ which converges to zero for $\alpha \leq 1-\frac{1}{2q}$ or $\beta \leq 1-\frac{1}{2q}$.  The bounds on the rate of convergence are given by the Lemmas \ref{lemma:estimateforT_1} and \ref{T_2}. Using Theorem \ref{th:NourdinPeccati}, the conclusion of the theorem follows.
\qed\end{proof}

\vskip0.5cm
The fact that the term $T_{1}$ converges to zero makes the difference between the situations treated in the above theorem and the non-central limit case proved in the next section.

\section{The Non Central Limit Theorem}
\label{section:non-central}

We will assume throughout this section that the Hurst parameters $\alpha , \beta$ satisfy
\begin{equation*}
1>\alpha , \beta >1-\frac{1}{2q}.
\end{equation*}
We will study the limit of the sequence $\tilde{V}_{N,M}$ given by the formula (\ref{tildeV}) with the renormalization factor $\varphi$ from Lemma \ref{T_2}, point 6.
Let us denote by $h_{N,M}$ the kernel of the random variable $\tilde{V}_{N,M}$ which is an element of the $q$th Wiener chaos, i.e.
\begin{eqnarray*}
h_{N,M}
&=& \frac{1}{q!} \varphi (\alpha ,\beta, N,M) \sum_{i=0}^{N-1} \sum_{j=0}^{M-1}  \mathbf{1}_{[\frac{i}{N}, \frac{i+1}{N}] \times [\frac{j}{M}, \frac{j+1}{M}] }^{\otimes q}.
\end{eqnarray*}
We will prove that $(h_{N,M})_{N,M\geq 1}$ is a Cauchy sequence in the Hilbert space $\left({\cal{H}}^{\alpha, \beta}\right) ^{\otimes q}$. Using relation (\ref{scaH}), we obtain
\begin{eqnarray*}
\langle h_{N,M}, h_{N',M'} \rangle _{\left({\cal{H}}^{\alpha, \beta}\right) ^{\otimes q}} &=& \frac{1}{q! ^{2}}\varphi(\alpha , \beta, N, M) \varphi (\alpha , \beta , N',M')\\
&&\times (\alpha (2\alpha -1) )^{q}\sum_{i=0}^{N-1} \sum_{i'=0}^{N'-1}  \left( \int_{\frac{i}{N}}^{\frac{i+1}{N}}\int_{\frac{i'}{N}}^{\frac{i'+1}{N}}\vert u-v\vert ^{2\alpha -2}dudv \right) ^{q} \\
&&\times (\beta (2\beta -1) )^{q} \sum_{j=0}^{M-1} \sum_{j'=1}^{M'-1} \left(\int_{\frac{j}{M}}^{\frac{j+1}{M}}
\int_{\frac{j'}{M'}}^{\frac{j'+1}{M'}}\vert u-v\vert ^{2\beta -2} dudv \right) ^{q}
\end{eqnarray*}
and this converges to (see also \cite{BretonNourdin} or \cite{TV})
\begin{equation*}
c_{2}(\alpha ,\beta )  \frac{1}{q! ^{2}}(\alpha (2\alpha -1) )^{q}(\beta(2\beta -1) )^{q} \int_{0}^{1} \int_{0}^{1} \vert u-v\vert ^{(2\alpha -2)q}dudv \int_{0}^{1} \int_{0}^{1} \vert u-v\vert ^{(2\beta -2)q}dudv,
\end{equation*}
where $c_{2}(\alpha , \beta)= \frac{q!}{\kappa _{\alpha } \kappa _{\beta}}$. The above constant is equal to
\begin{equation*}
c_{2}(\alpha ,\beta ) \frac{1}{q! ^{2}}(\alpha (2\alpha -1) )^{q}(\beta (2\beta -1) )^{q} \frac{1}{(\alpha q-q+1)(2\alpha q-2q+1)}\frac{1}{(\beta q-q+1)(2\beta q-2q+1)}.
\end{equation*}
\begin{remark}
Note that the above constant is actually $\frac{1}{q!}$.
\end{remark}
It follows that the sequence $h_{N,M}$ is Cauchy in the Hilbert space $\left({\cal{H}}^{\alpha, \beta}\right) ^{\otimes q}$ and as $N,M\to \infty$ it has a limit in $\left({\cal{H}}^{\alpha, \beta}\right) ^{\otimes q}$ denoted by $\mu  ^{(q)} $.
In the same way, the sequence
\begin{eqnarray*}
h_{N,M}(t,s)
&=& \frac{1}{q!} \varphi(\alpha, \beta ,N,M) \sum_{i=0}^{[(N-1)t]} \sum_{j=0}^{[(M-1)s]}  \mathbf{1}_{[\frac{i}{N}, \frac{i+1}{N}] \times [\frac{j}{M}, \frac{j+1}{M}] }^{\otimes q}
\end{eqnarray*}
is Cauchy in $\left({\cal{H}}^{\alpha, \beta}\right) ^{\otimes q}$ for every fixed $s,t$ and it has a limit in this Hilbert space which will be denoted by $\mu ^{(q)}_{s,t}$. Notice that $\mu ^{(q)}= \mu ^{(q)}(1,1)$ and that $\mu ^{(q)}$  is a normalized uniform measure on the set $([0,t] \times [0,s] ) ^{q}$.
\begin{definition}\label{hermitesheet}
We define the Hermite sheet process of order $q$ and with Hurst parameters $\alpha, \beta \in (0,1)$, denoted by $(Z^{(q)}_{t,s})_{s,t\in [0,1]}$,  by
\begin{equation*}
Z^{(q), \alpha, \beta}_{t,s}:=Z^{(q)}_{t,s} =I_{q} (\mu ^{(q)}_{s,t}), \hskip0.5cm \forall s,t \in [0,1].
\end{equation*}
\end{definition}
The previous computations lead to the following theorem.
\begin{theorem}
Let $\tilde{V}_{N,M}$ be given by (\ref{tildeV}) with the function $\varphi $ defined in Lemma  \ref{T_2}, point 6. Consider the Hermite sheet introduced in Definition \ref{hermitesheet}. Then for $q\geq 2$ it holds
$$ \lim_{N,M\to \infty} \mathbb{E}[|\tilde{V}_{N,M}-Z|^2]=0,$$
where $Z:=Z^{(q)}_{1,1}$.
\end{theorem}
\begin{proof}
Note that $\frac{1}{q!}\mathbb{E}[|\tilde{V}_{N,M}-Z|^2]=\|h_{N,M}\|_{(\mathcal{H}^{\alpha,\beta})^{\otimes q}}^2 + \|\mu^{(q)}\|_{(\mathcal{H}^{\alpha,\beta})^{\otimes q}}^2 - 2 \langle h_{N,M}, \mu^{(q)}\rangle_{(\mathcal{H}^{\alpha,\beta})^{\otimes q}}$. The computations of the beginning of this section complete the proof.
\qed\end{proof}

\vskip0.3cm

Let us prove below some basic properties of the Hermite sheet.
\begin{prop}Let us consider the Hermite sheet $(Z^{(q)}_{s,t})_{s,t\in [0,1]}$ from Definition \ref{hermitesheet}. We have the following:
\begin{description}
\item{a) } The covariance of the Hermite sheet is given by
\begin{equation*}
\mathbb{E} Z^{(q)}_{s,t} Z^{(q)} _{u,v} = R_{q(\alpha-1)+1 } (s,u) R_{q(\beta-1)+1 } (t,v).
\end{equation*}
Consequently, it has the same covariance as the fractional Brownian sheet with Hurst parameters $q(\alpha-1)+1$ and $q(\beta-1)+1$.
\item{b) } The Hermite process is self-similar in the following sense: for every $c,d>0$,  the process
\begin{equation*}
\hat{Z} ^{(q)}_{s,t}:= (Z^{(q)})_{cs, dt}
\end{equation*}
has the same law as $c^{q(\alpha -1)+1} d^{q(\beta -1)+1}Z^{(q)}_{s,t}$.
\end{description}
\item{c) } The Hermite process has stationary increments in the sense of  Definition \ref{stationary}.
\item{d) } The paths are H\"older continuous of order $(\alpha ', \beta ') $ with $0<\alpha '<\alpha $ and $0<\beta '<\beta$.
\end{prop}
\begin{proof}
Let $f$ be an arbitrary function in $\left( {\cal{H}}^{\alpha , \beta }\right) ^{\otimes q}$. It holds that
\begin{eqnarray*}
&&\langle h_{N,M}(t,s) , f\rangle _{\left( {\cal{H}}^{\alpha , \beta }\right) ^{\otimes q}} \\
&=&c_{2}(\alpha , \beta ) ^{-1/2}\frac{N^{q-1}M^{q-1}}{q!} \sum_{i=0}^{[(N-1)t]} \sum_{j=0}^{[(M-1)s]}  \langle \mathbf{1}_{[\frac{i}{N}, \frac{i+1}{N}] \times [\frac{j}{M}, \frac{j+1}{M}] }^{\otimes q}, f\rangle  _{\left( {\cal{H}}^{\alpha , \beta }\right) ^{\otimes q}}\\
&=&a(\alpha )^{q}a(\beta)^{q}c_{2}(\alpha , \beta ) ^{-1/2}\frac{N^{q-1}M^{q-1}}{q!} \sum_{i=0}^{[(N-1)t]} \sum_{j=0}^{[(M-1)s]} \\
&&\int_{[0,1] ^{2q}}dx_{1}..dx_{q}dy_{1}..dy_{q} f\left( (x_{1},y_{1}),..,(x_{q}, y_{q})\right) \times \int_{[\frac{i}{N}, \frac{i+1}{N}]^{q}}da_{1}..da_{q}\int_{ [\frac{j}{M}, \frac{j+1}{M}]^{q}}db_{1}..db_{q}\\
&&\times \prod_{k=1}^{q} \vert a_{k}-x_{k}\vert ^{2\alpha -2}\prod_{k=1}^{q}\vert b_{k}-y_{k} \vert ^{2\beta -2}\\
&\overset{N,M\to\infty}{\to}& a(\alpha )^{q}a(\beta)^{q}c_{2}(\alpha , \beta ) ^{-1/2}\frac{1}{q!} \int_{0}^{t}da\int_{0}^{s}db \int_{[0,1] ^{2q}}dx_{1}..dx_{q}dy_{1}..dy_{q}\\
&&\times f\left( (x_{1},y_{1}),..,(x_{q}, y_{q})\right)\prod_{k=1}^{q} \vert a-x_{k}\vert ^{2\alpha -2} \prod_{k=1}^{q}\vert b-y_{k} \vert ^{2\beta -2}.
\end{eqnarray*}
By applying the above formula for $f=\mu ^{(q)}_{u,v}$ and using the fact that
\begin{equation*}
\mathbb{E}\left( Z^{(q)}_{s,t} Z^{(q)}_{u,v}\right) =q! \langle \mu ^{(q)}_{s,t}, \mu^{(q)}_{u,v}\rangle _{\left( {\cal{H}}^{\alpha , \beta }\right) ^{\otimes q}},
\end{equation*}
we obtain the point a).\\
Concerning b), let us denote by
\begin{equation*}
H_{N,M}(t,s)=\frac{1}{q!} \varphi(\alpha,\beta,N,M)\sum_{i=0} ^{[(N-1)t]}\sum_{j=0}^{[(M-1)]s} I_{q} \left(\mathbf{1}_{[\frac{i}{N}, \frac{i+1}{N}]\times [\frac{j}{M}, \frac{j+1}{M}]} ^{\otimes q} \right).
\end{equation*}
We know that
\begin{equation}\label{n1}
H_{cN, dM}(t,s)\overset{N,M\to \infty}{\to} Z^{(q)}_{t,s}
\end{equation}
in $L^{2}(\Omega)$ for every $s,t\in [0 ,1]$. But
\begin{eqnarray}
H_{cN, dM}(t,s)&=&\frac{c_{2}(\alpha , \beta ) ^{-1/2}}{(cN)^{1-(1-\alpha )q}(dM)^{1-(1-\beta) q}} \sum_{i=0}^{[(N-1)t]} \sum_{j=0}^{[(M-1)s]} I_{q} \left(\mathbf{1}_{[\frac{i}{N}, \frac{i+1}{N}]\times [\frac{j}{M}, \frac{j+1}{M}]} ^{\otimes q} \right) \nonumber\\
&=& \frac{1}{(c)^{1-(1-\alpha )q}(d)^{1-(1-\beta) q}}H_{N,M}(t,s) \overset{N,M\to \infty}{\to}  Z^{(q)} _{t,s}\label{n2}.
\end{eqnarray}
The point b) follows easily from (\ref{n1}) and (\ref{n2}).\\
Point c) is a consequence of the fact that the fractional Brownian sheet has stationary increments in the sense of Definition \ref{stationary} while point d) can be easily proved by using  Kolmogorov continuity criterion together with points b) and c) above (see also Section 4, page 35-36 in \cite{AyacheLegerPontier2002}).
\qed\end{proof}

\section{Appendix}
We recall the following two technical lemmas which have been proved in \cite{NoPe1} and \cite{BretonNourdin}.

\begin{lemma}
\label{lemma:estimqteforT_2} Let $\gamma$ in $(0,1)$ and $q$ be an
integer with $q\geq 2$. We set
$$ r_\gamma(z):=\frac12 \left(|z+1|^{2\gamma}+|z-1|^{2\gamma}-2|z|^{2\gamma}\right), \quad z \in \mathbb{Z}.$$
We have:
\begin{itemize}
\item[(i)] If $0<\gamma<1-\frac{1}{2q}$, then
$$ \lim_{N\to \infty} N^{2\gamma q-1} \sum_{i,i'=0}^{N-1} \langle \mathbf{1}_{\Delta i}, \mathbf{1}_{\Delta i'}\rangle_{{\cal{H}}^{\gamma} }^{q} = \sum_{r\in \mathbb{Z}} r_\gamma(z)^q=:s_\gamma,$$
and $\vert N^{2\gamma q-1} \sum_{i,i'=0}^{N-1} \langle \mathbf{1}_{\Delta i},
\mathbf{1}_{\Delta i'}\rangle_{{\cal{H}}^{\gamma} }^{q}-s_\gamma\vert
\unlhd N^{-1}+N^{2q\gamma-2q+1}$.
\item[(ii)] If $\gamma=1-\frac{1}{2q}$, then
$$\lim_{N\to \infty} \log(N)^{-1} N^{2 q-2} \sum_{i,i'=0}^{N-1} \langle \mathbf{1}_{\Delta i}, \mathbf{1}_{\Delta i'}\rangle_{{\cal{H}}^{\gamma} }^{q}=2 \left(\frac{(2q-1)(q-1)}{2q^2}\right)^q =:\iota_\gamma,$$
and $\vert \log(N)^{-1} N^{2 q-2} \sum_{i,i'=0}^{N-1} \langle
\mathbf{1}_{\Delta i}, \mathbf{1}_{\Delta i'}\rangle_{{\cal{H}}^{\gamma}
}^{q}-\iota_\gamma \vert \unlhd \log(N)^{-1}$.
\item[(iii)] If $\gamma>1-\frac{1}{2q}$, then
$$\lim_{N\to \infty} N^{2q-2} \sum_{i,i'=0}^{N-1} \langle \mathbf{1}_{\Delta i}, \mathbf{1}_{\Delta i'}\rangle_{{\cal{H}}^{\gamma} }^{q}= \frac{\gamma^q(2\gamma-1)^q}{(\gamma q-q+1)(2\gamma q-2q+1)}=:\kappa_\gamma,$$
and $\vert N^{2q-2} \sum_{i,i'=0}^{N-1} \langle \mathbf{1}_{\Delta i},
\mathbf{1}_{\Delta i'}\rangle_{{\cal{H}}^{\gamma} }^{q}-\kappa_\gamma
\vert \unlhd N^{2q-1-2 \gamma q}$.
\end{itemize}
\end{lemma}
\begin{proof}
The first two claims can be found respectively in
\cite[p.~102]{NoPe1}, \cite[p.~491-492]{BretonNourdin}. For the
third part we define $f_N:=N^{q-1}\sum_{k=0}^{N-1}
\mathbf{1}_{[\frac{k}{N},\frac{k+1}{N}]}^{\otimes q}$. Then $f_N$ is a Cauchy
sequence in $(\mathcal{H}^\gamma)^{\otimes q}$ with limit $f$ and
$\Vert f \Vert^2_{(H^\gamma)^{\otimes q}}=\kappa_\gamma.$ For the
rate of convergence we have
\begin{eqnarray*}
\Vert f_N \Vert^2_{(H^\gamma)^{\otimes q}}-\Vert f \Vert^2_{(H^\gamma)^{\otimes q}}&=&\Vert f_N-f \Vert^2_{(H^\gamma)^{\otimes q}}+2\langle f_N-f,f\rangle_{(H^\gamma)^{\otimes q}}\\
&\leq& \Vert f_N-f \Vert^2_{(H^\gamma)^{\otimes q}} +2\Vert f_N-f
\Vert^2_{(H^\gamma)^{\otimes q}}\Vert f \Vert^2_{(H^\gamma)^{\otimes
q}}.
\end{eqnarray*}
Refer to  \cite[Proposition 3.1]{BretonNourdin} to see the details and to get that the order is $O(N^{2q-1-2 \gamma q})$ (a direct argument as in the proof of the next lemma can be also employed).\\\\
\qed\end{proof}
We also state the following estimates which have been
obtained respectively in \cite[p.~102, p.~104]{NoPe1} and in
\cite[p.~491-492]{BretonNourdin}.
\begin{lemma}
\label{lemma:estimateforT_1} Let $\gamma$ in $(0,1)$. We set
$q,p,a,b,c,d$ integers such that: $q\geq 2$, $p \in \{0,\ldots,
q-2\}$ and $a+b=c+d=q-1-p$. We have:
\begin{itemize}
\item[(i)] If $0<\gamma<1-\frac{1}{2q}$, then
\begin{eqnarray*}
&&N^{4q\gamma-2} \sum_{i,i',k,k'=0}^{N-1} \langle \mathbf{1}_{\Delta i}, \mathbf{1}_{\Delta i'}\rangle_{{\cal{H}}^{\gamma} }^{p+1} \langle \mathbf{1}_{\Delta k}, \mathbf{1}_{\Delta k'}\rangle_{{\cal{H}}^{\gamma} }^{p+1} \langle \mathbf{1}_{\Delta i}, \mathbf{1}_{\Delta k}\rangle_{{\cal{H}}^{\gamma} }^{a} \langle \mathbf{1}_{\Delta i'}, \mathbf{1}_{\Delta k'}\rangle_{{\cal{H}}^{\gamma} }^{b}\\
&&\quad \quad \quad \times \langle \mathbf{1}_{\Delta i'}, \mathbf{1}_{\Delta k}\rangle_{{\cal{H}}^{\gamma} }^{c} \langle \mathbf{1}_{\Delta i'}, \mathbf{1}_{\Delta k'}\rangle_{{\cal{H}}^{\gamma} }^{d}\\
&\unlhd& N^{-1}+N^{2\gamma-2}+N^{2\gamma q -2q+1}.
\end{eqnarray*}
\item[(ii)] If $\gamma=1-\frac{1}{2q}$, then
\begin{eqnarray*}
&&\frac{N^{4q-4}}{\log(N)^2} \sum_{i,i',k,k'=0}^{N-1} \langle \mathbf{1}_{\Delta i}, \mathbf{1}_{\Delta i'}\rangle_{{\cal{H}}^{\gamma} }^{p+1} \langle \mathbf{1}_{\Delta k}, \mathbf{1}_{\Delta k'}\rangle_{{\cal{H}}^{\gamma} }^{p+1} \langle \mathbf{1}_{\Delta i}, \mathbf{1}_{\Delta k}\rangle_{{\cal{H}}^{\gamma} }^{a} \langle \mathbf{1}_{\Delta i'}, \mathbf{1}_{\Delta k'}\rangle_{{\cal{H}}^{\gamma} }^{b}\\
&&\quad \quad \quad \times \langle \mathbf{1}_{\Delta i'}, \mathbf{1}_{\Delta k}\rangle_{{\cal{H}}^{\gamma} }^{c} \langle \mathbf{1}_{\Delta i'}, \mathbf{1}_{\Delta k'}\rangle_{{\cal{H}}^{\gamma} }^{d}\\
&\unlhd&\log(N)^{-1}.
\end{eqnarray*}
\item[(iii)] If $\gamma>1-\frac{1}{2q}$, then
\begin{eqnarray*}
&&N^{4q-4} \sum_{i,i',k,k'=0}^{N-1} \langle \mathbf{1}_{\Delta i}, \mathbf{1}_{\Delta i'}\rangle_{{\cal{H}}^{\gamma} }^{p+1} \langle \mathbf{1}_{\Delta k}, \mathbf{1}_{\Delta k'}\rangle_{{\cal{H}}^{\gamma} }^{p+1} \langle \mathbf{1}_{\Delta i}, \mathbf{1}_{\Delta k}\rangle_{{\cal{H}}^{\gamma} }^{a} \langle \mathbf{1}_{\Delta i'}, \mathbf{1}_{\Delta k'}\rangle_{{\cal{H}}^{\gamma} }^{b}\\
&&\quad \quad \quad \times \langle \mathbf{1}_{\Delta i'}, \mathbf{1}_{\Delta k}\rangle_{{\cal{H}}^{\gamma} }^{c} \langle \mathbf{1}_{\Delta i'}, \mathbf{1}_{\Delta k'}\rangle_{{\cal{H}}^{\gamma} }^{d}\\
&\unlhd &1.
\end{eqnarray*}
\end{itemize}
\end{lemma}
\begin{proof}
The first point is proved in \cite[p.~102, 105]{NoPe1}. The point (ii) is done in \cite[p.~491-492]{BretonNourdin}. The last case can be treated in the following way. The quantity  $\langle \mathbf{1}_{\Delta i}, \mathbf{1}_{\Delta i'}\rangle_{{\cal{H}}^{\gamma} }$ is equivalent with a constant times $N^{-2\gamma} \vert i-i'\vert ^{2\gamma -2} $ and the sum appearing in (iii) is then equivalent to
\begin{eqnarray*}
&&N^{4q-4}N^{-4\gamma q}   \sum_{i,i',k,k'=0}^{N-1}\vert i-i'\vert ^{(2\gamma -2)(p+1) }\vert k-k'\vert ^{(2\gamma -2)(p+1) }\\
&&\times \vert i-k\vert ^{(2\gamma -2)a }\vert i'-k'\vert ^{(2\gamma -2)b }\vert i'-k\vert ^{(2\gamma -2)c }\vert i-k'\vert ^{(2\gamma -2)d }\\
&=&N^{-4}  \sum_{i,i',k,k'=0}^{N-1}N^{-2q(2\gamma-2)}\vert i-i'\vert ^{(2\gamma -2)(p+1) }\vert k-k'\vert ^{(2\gamma -2)(p+1) }\\
&&\times \vert i-k\vert ^{(2\gamma -2)a }\vert i'-k'\vert ^{(2\gamma -2)b }\vert i'-k\vert ^{(2\gamma -2)c }\vert i-k'\vert ^{(2\gamma -2)d },
\end{eqnarray*}
for $N$ large enough and this is a Riemann sum which converges to a constant.
\qed\end{proof}

\section*{Acknowledgments}
 Anthony R\'eveillac is grateful to DFG Research center Matheon project E2 for financial support. Ciprian Tudor would like to acknowledge generous support from the Alexander von Humboldt Foundation which made possible several research visits at the Humboldt Universit\"at zu Berlin.

\end{document}